\documentclass[11pt,a4paper,reqno]{amsart}

\usepackage{graphics,float}
\usepackage[latin1]{inputenc} %Caractï¿½res accentuï¿½s - iso 8859-1

\usepackage{graphicx} 

\usepackage[pdftex]{hyperref} %genial: permet les liens hypertextes dans le pdf

\usepackage{amsfonts,latexsym,amssymb} %paquets mathï¿½matiques
\usepackage{mathrsfs,amsmath}
\usepackage{amsthm} %JBP: rajouté
\usepackage[all]{xy}

\usepackage{color} %couleurs, non prï¿½sent dans MikTex ?

\DeclareMathOperator\Rank{Rank}
\newcommand{\N}{\ensuremath{\mathbb{N}}}
\newcommand{\R}{\ensuremath{\mathbb{R}}}

\renewcommand{\phi}{\varphi}

\newcommand{\vecA}{\mathbf{A}}
\newcommand{\vecB}{\mathbf{B}}
\newcommand{\vecC}{\mathbf{C}}
\newcommand{\vecD}{\mathbf{D}}

\newcommand{\vecF}{\mathbf{F}}

\newcommand{\vecH}{\mathbf{H}}
\newcommand{\vecI}{\mathbf{I}}

\newcommand{\vecM}{\mathbf{M}}

\newcommand{\vecT}{\mathbf{T}}

\newcommand{\vecX}{\mathbf{X}}

\newcommand{\vecCijbeta}[3]{\vecC_{#1,#2}^{#3}}

\newcommand{\vece}{\mathbf{e}}
\newcommand{\vecf}{\mathbf{f}}

\newcommand{\vecs}{\mathbf{s}}
\newcommand{\vecu}{\mathbf{u}}
\newcommand{\vecv}{\mathbf{v}}

\newcommand{\vecx}{\mathbf{x}}
\newcommand{\vecy}{\mathbf{y}}
\newcommand{\vecz}{\mathbf{z}}

    \newtheorem{theorem}{Theorem}[section]
    \newtheorem{lemma}[theorem]{Lemma}
    \newtheorem{proposition}[theorem]{Proposition}
    \newtheorem{definition}[theorem]{Definition} %[section]
   \newtheorem{hypo}[theorem]{Assumption}%[section]
    \newtheorem{remark}[theorem]{Remark}%[section]

\usepackage[foot]{amsaddr}

\title[Controllability of the 2-link magneto-elastic micro-smimmer]{\bf Local Controllability of the Two-link Magneto-elastic Micro-swimmer}

\author{Laetitia Giraldi \qquad Jean-Baptiste Pomet}
\address{Universit{\'e} C{\^o}te d'Azur, Inria, CNRS, LJAD, France}
\thanks{L.~Giraldi was partially funded by the labex LMH through the grant ANR-11-LABX-0056-LMH}
\email{laetitia.giraldi@inria.fr, jean-baptiste.pomet@inria.fr}

\date{to appear in \textit{IEEE Trans. Autom. Control}, 2017
  (electronically: 2016). \\DOI: \href{http://dx.doi.org/10.1109/TAC.2016.2600158}{\textcolor{blue}{10.1109/TAC.2016.2600158}}}

\begin{document}
\maketitle

\begin{abstract}
  A recent promising technique for robotic micro-swimmers is to endow
  them with a magnetization and apply an external magnetic field to
  provoke their deformation. In this note we consider a simple planar
  micro-swimmer model made of two magnetized segments connected by an
  elastic joint, controlled via a magnetic field. After recalling the
  analytical model, we establish a local controllability result around
  the straight position of the swimmer.\\
 % \smallskip
%\noindent
\end{abstract}

\section{Introduction}

Micro-scale robotic swimmers have potential high impact
applications. For instance, they could be used in new therapeutic and
diagnostic procedures such as targeted drug delivery or minimized
invasive microsurgical operations \cite{NelsonKaliakatsos10}.  One of
the main challenges is to design a controlled micro-robot able to swim
through a narrow channel.  In this context, This note states a
controllability result for a (simplified model of) micro-robot which
is controlled by an external magnetic field that provokes deformations
on the magnetic body.

Swimming is the ability for a body to move through a fluid by
performing self-deformations.  In general, the fluid is modeled by
Navier-Stokes equations and the coupling with the swimming body gives
rise to a very complex model.  It is well known \cite{Purcell77} that
the locomotion of microscopic bodies in fluids like water (or of
macroscopic bodies in a very viscous fluid) is characterized by a low
Reynolds number (a dimensionless ratio between inertia effects and
viscous effects, of the order of $10^{-6}$ for common
micro-organisms), and that it is then legitimate to consider that the
fluid is governed by the Stokes equations, so that hydrodynamic forces
applied to the swimmer can be derived linearly with respect to its
speed, i.e., the associated Dirichlet-to-Neumann mapping is linear,
see for more details
\cite{AlougesGiraldi12,BrennerHappel65,PowersLauga09}.

The absence of inertia in the model, that we assume from now on, makes
mobility more difficult: a typical obstruction, known as the scallop
theorem \cite{Purcell77}, imposes to use non reciprocal swimming
strategies for achieving their self-propulsion.  It means that the
swimmer is not able to move by using a periodic change of shape with
only one degree of freedom.

The Resistive Force Theory \cite{GrayHancock55} consists in
approximating the Dirichlet-to-Neumann mapping by an explicit
dependence of the hydrodynamic force on the relative speed at each
point of the boundary.  Following \cite{AlougesDeSimone14,Or14}, the
presnt note uses the simplified model resulting from this approximation.

As far as controllability of these devices is concerned, most known
results, starting with the pioneering work of A. Shapere and
F. Wilczek \cite{Shapere89}, focus on the case where the control is
the rate of deformation of the swimmer's shape.  One then deals with a
driftless control-affine system and controllability derives, at least
when the shape has a finite number of degrees of freedom, from Lie
algebraic methods, see for instance in \cite{Coron56}; more generaly
these control problems are related to non-holonomy of distributions
and sub-Riemannian geometry \cite{Montgomery02}, where not only
controllability but also optimality may be studied, as in
\cite{BettiolBonnard15,cdc2013} where the shape has two degrees of
freedom and equations derive from the Resistive Force Theory
(Purcell's swimmer).  A model with more degrees of freedom is
considered in \cite{AlougesDeSimone13}.  In
\cite{AlougesGiraldi12,Gerard-VaretGiraldi13,LoheacScheid11,ChambrionGiraldi14},
the model is more complex, based on explicit solutions of Stokes
equations and not on the Resistive Force Theory.

In \cite{AlougesDeSimone14}, for the first time, the case where the
filament is not fully controlled was considered: the filament is
discretized into line segment with a certain magnetization; the
elasticity is represented by a torsional spring at each joint and the
control is, instead of the rate of deformation, an external magnetic
field that provokes deformation (and also some movement) of the
magnetized shape.  In that paper, by exploiting sinusoidal external
magnetic fields, the authors show numerically that the swimmer can be
steered along one direction by using prescribed sinusoidal magnetic
field.
Here we simplify the latter model by considering a swimmer made of
only two magnetized segments connected by an elastic joint.  This
reduced model has already been addressed in \cite{Or14}, in which the
authors studied the effect of a prescribed sinusoidal magnetic field
by expressing the leader term of the displacement with respect to
small external magnetic fields.

We go further by stating a local controllability result including
reorientation and not only movement along the longitudinal direction.
This is, to our knowledge, the first true local controllability result
for these partially actuated systems.  We point certain degenerate
values of the model's parameters (lengths, magnetizations, spring
constant) for which controllability does not hold and prove local
controllability for other values.

The model is derived in Section \ref{sec:Modeling}; this part only
uses ideas taken from \cite{AlougesDeSimone14}, but is needed to
obtain more precise expressions of the equations of motion.  The
controllability results are stated, commented and proved in
Section~\ref{sec:MainResult}. Finally, Section \ref{sec:Perspectives}
briefly states perspectives of this result.

\section{Modeling}
\label{sec:Modeling}

The present note considers the same swimmer model as \cite{Or14}; it
consists of 2 magnetized segments of length $\ell_1$ and $\ell_2$,
with a magnetic moment $M_1$ and $M_2$ respectively, connected by a
joint equipped with a torsional spring of stiffness $\kappa$ that
tends to align the segments with one another.  It is constrained to
move in a plane and is subject to an external magnetic field $\vecH$
as well as hydrodynamic forces due to the ambient fluid (these are
characterized later).
\begin{figure}[h]
  \begin{center}
  \includegraphics[scale=1]{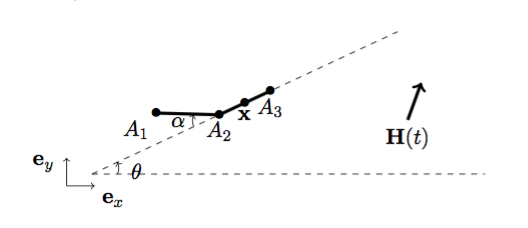}
%    \begin{tikzpicture}[scale=0.5,rotate=25]
%      \draw[dashed] (-6,0)--(6,0); \draw[ultra thick] (-1,0) -- (1,0);
%      \draw[ultra thick] (-1,0) -- ( -3,1);
%      \draw [dashed] (-6,0) -- (6.68,-5.91); \draw [->](-7,0) -- (-7
%      +0.906,-0.422); \draw [->] (-7,0) -- (-7 +0.422, 0.906); \draw
%      [->][ultra thick] (5,-3) -- (6,-2);
%      \draw (0,0) node {$\bullet$}; \draw (-1,0) node{$\bullet$};
%      \draw (1,0) node{$\bullet$}; \draw (-3,1) node{$\bullet$}; \draw
%      [->] (-2,0) arc (0:-50:-0.5); \draw [<-] (-5,0) arc (0:-50:0.5);
%      \draw (-3,1) node[below left] {$\tiny{A_1}$}; \draw (-1,0)
%      node[below] {$\tiny{A_2}$}; \draw (1,0) node[below]
%      {$\tiny{A_3}$}; \draw (0,0) node[below] {$\tiny{\vecx}$}; \draw
%      (-2.6,0) node[above] {$\alpha$}; \draw (-4.7,0.1) node[below
%      right] {$\theta$}; \draw (-7 +0.906,-0.422) node[below right]
%      {$\tiny{\vece_x}$}; \draw (-7 +0.422, 1.3)node[below left]
%      {$\tiny{\vece_y}$}; \draw (5,-3) node[below]
%      {$\tiny{\vecH(t)}$}; 
%\end{tikzpicture}
\caption{Magneto-elastic $2$-link swimmer in plane subject to an external magnetic field $\vecH(t)$.}
\label{2_links_swimmer}
\end{center}
\end{figure}

Under the low Reynolds number assumption introduced and justified in
the introduction, inertia is negligible and the equations are
obtained, as in \cite{BrennerHappel65,PowersLauga09}, by a simple
balance of forces and torques (see \cite{AlougesDeSimone14}):
\begin{equation}
\label{eq:sum_of_forces} 
\left\{
\begin{array}{lcc}
\vecF_1^h+\vecF_2^h &=& 0\,,\\
\vecT_1^h+\vecT^h_2 +\vecT_1^m+\vecT_2^m&=&0\,,\\
\vecT_2^h+\vecT_2^m + \vecT_2^{el} &=& 0\,,
\end{array} \right.
\end{equation}
where
\begin{itemize}
\item the two first equations state that the balance of exterior
forces and exterior moments for the whole system is zero, denoting by
$\vecF_i^h$ the hydrodynamic force applied by the fluid on the
$i$\textsuperscript{th} link, and by $\vecT_i^h$ and $\vecT_i^m$ the
moment (with respect to any point, so we chose it to be $A_2$) applied
on the $i$\textsuperscript{th} link by the fluid and the magnetic
field respectively,
\item the last equation states that the moment with respect to $A_2$
of the forces and torques applied the subsystem consisting of the
second link $[A_2,A_3]$ is zero, $\vecT_2^{el}$ being the elastic
torque applied by the first link on the second one (the moment of the
contact force is zero because it is applied at $A_2$).
\end{itemize} This only contains four non-trivial relations because
the first equation takes place in the horizontal plane and the other
two on the vertical axis.

Let $\left(\vece_x,\vece_y\right)$ be a fixed frame spanning the
$2$d-plane in which the robot evolves and set $\vece_z:=\vece_x \times
\vece_y$. We call $\vecx=(x,y)$ the coordinates in the frame
$\left(\vece_x,\vece_y\right)$ of the central point of the second
segment, $\theta$ the angle that it forms with the $x$-axis, $\alpha$
the relative angle between the first and second segments (see Figure
\ref{2_links_swimmer}). The position and orientation of the swimmer
are characterized by the triplet $(x,y,\theta)$, and its shape by
$\alpha$.  We denote by
$$
\vece_{1,\parallel} = \begin{pmatrix} \cos(\theta+\alpha)\\
\sin(\theta+\alpha)\end{pmatrix}\,, \quad \vece_{2,\parallel}
= \begin{pmatrix} \cos\theta\\ \sin\theta\end{pmatrix}\,
$$
the unit vectors aligned with segments $[A_1,A_2]$ and $[A_2,A_3]$,
their orthogonal vectors by
$$
\vece_{1,\perp} = \begin{pmatrix} -\sin(\theta+\alpha)\\
\cos(\theta+\alpha)\end{pmatrix}\,, \quad \vece_{2,\perp}
= \begin{pmatrix} -\sin\theta\\ \cos\theta\end{pmatrix}\,\,.
$$
We assume that the external magnetic field $\vecH$ is horizontal in
such a way that the motion holds in the plane generate by $\vece_x$
and $\vece_y$ and we call $H_{\parallel}$ and $H_{\perp}$ its
coordinates into the moving frame
 $$\vecH=H_{\parallel}\,\vece_{2,\parallel}+H_{\perp}\,\vece_{2,\perp}.$$
 
Let us now compute the different contributions in
\eqref{eq:sum_of_forces}.

\medskip

\paragraph{\bf Elastic effects}
The torsional spring delivers the following torque to segment $[A_2,A_3]$:
$$
\vecT_2^{el} = \kappa \,\alpha\, \vece_z\,.
$$

\medskip

\paragraph{\bf Magnetic effects} 
The magnetic torque applied to the
$i$\textsuperscript{th} segment is 
$$
\vecT^{m}_{i} = M_i \left(\vece_{i,\parallel} \times \vecH\right)\,,
\quad i=1,2\,,
$$
with the notations defined above. 

\medskip

\paragraph{\bf Hydrodynamic effects} The force applied to the swimmer by
the fluid depends on their relative speed.  As announced in the
introduction, we use the Resistive Force Theory \cite{GrayHancock55},
that assumes that the hydrodynamic drag force of each segment is
linear with respect to its velocity.  More precisely, if the point of
abscissa $s$ on segment $i$ ($i=1,2$) has velocity $\vecu_i(s) =
u_{i,\parallel}(s) \, \vece_{i,\parallel}+u_{i,\perp}(s) \,
\vece_{i,\perp}$, then the drag force applied to that point is given
by
\begin{equation}
\label{eq_local_forceBIS} \vecf_{i}(s) = -\,\xi_i\,
u_{i,\parallel}(s)\, \vece_{i,\parallel} \;-\;\eta_i \,
u_{i,\perp}(s)\, \vece_{i,\perp} \,,\ \ i=1,2\,,
\end{equation} with $\xi_i$ and $\eta_i$ the constant positive drag
coefficients. % , satisfying $\xi_i\leq\eta_i$ in general.

Denote by $\mathbf{R}_\varphi=\left(\begin{smallmatrix}
\cos\varphi&-\sin\varphi\\ \sin\varphi & \cos\varphi
\end{smallmatrix}\right)$ the matrix of the rotation of angle
$\varphi$, for any $\varphi$.  The matrix $\mathbf{R}_{\theta+\alpha}$
sends the basis $\left(\vece_x,\vece_y\right)$ onto the basis
$(\vece_{1,\parallel},\vece_{1,\perp})$ and $\mathbf{R}_\theta$ sends
the basis $\left(\vece_x,\vece_y\right)$ onto the basis
$(\vece_{2,\parallel},\vece_{2,\perp})$, hence relation
\eqref{eq_local_forceBIS} translates into
\begin{equation}
\label{eq_local_force}
\begin{array}{l} \vecf_1(s) = - \mathbf{R}_{(\theta+\alpha)}\,\vecD_1
\,\mathbf{R}_{-(\theta+\alpha)} \,\vecu_1(s)\,, \\ \vecf_2(s) =-
\mathbf{R}_{\theta} \,\vecD_2 \, \mathbf{R}_{-\theta} \,\vecu_2(s)\,,
\end{array}
\end{equation} with $\vecD_i$ the matrix $\left(\begin{smallmatrix}
\xi_i&0\\ 0 &\eta_i
\end{smallmatrix}\right)$ for $i=1,2$ and where the vectors are
coordinates in the bases $\left(\vece_x,\vece_y\right)$.

If the origin of the abscissa $s$ is the point $A_2$ on segment 1 and
the point $\vecx$ on segment 2, one has
\begin{align}
  \label{eq:1} &\vecx_1(s)
=\vecx-\frac{\ell_2}{2}\,\vece_{2,\parallel}\,-s\,\vece_{1,\parallel}\,,
&0\leq s\leq \ell_1\,,\\
  \label{eq:2}
&\vecx_2(s)=\vecx+s\,\vece_{2,\parallel}\,,&\hspace*{-20em}-\frac{\ell_2}{2}\leq
s\leq\frac{\ell_2}{2}\,,
\end{align}
hence
\begin{align}
  \label{eq:6}
 &\vecu_1(s) = \dot\vecx - \frac{\ell_2}{2} \dot\theta\vece_{2,\perp}
-s \left(\dot\alpha+\dot\theta\right) \vece_{1,\perp}\,, &0\leq s\leq
\ell_1\,,\\
  \label{eq:7}
&\vecu_2(s) = \dot\vecx + s \dot\theta\vece_{2,\perp}\,,
&\hspace*{-2em}-\frac{\ell_2}{2}\leq s\leq\frac{\ell_2}{2}\,.
\end{align}

The total hydrodynamic force acting on the first and second segments
are given by $\vecF^h_1 = \int_0^{\ell_1} \vecf_1(s) \mbox{d}s$,
$\vecF_2^h = \int_{-\frac{\ell_2}{2}}^{\frac{\ell_2}{2}} \vecf_2(s)
\mbox{d}s$, and a straightforward integration yields
\begin{align}
\label{eq:Force1} 
&\vecF^h_1 = \mathbf{R}_{\theta+\alpha} \vecD_1
\left(-\ell_1\mathbf{R}_{-(\theta+\alpha)}\,\dot \vecx
+\frac{\ell_1\ell_2}{2} \dot{\theta}\,\mathbf{R}_{-\alpha}\vece_y
+\frac{{\ell_1}^2}{2}\left(\dot\theta+\dot\alpha\right)\vece_y\right),
\\ &\vecF_2^h = -\ell_2\mathbf{R}_{\theta} \vecD_2
\mathbf{R}_{-\theta} \dot\vecx\,.
\label{eq:Force2}
\end{align}
 
The moment with respect to point $A_2$ of the hydrodynamic forces
generated by the $i$\textsuperscript{th} segment has the expression
$$
\vecT_i^h := \int_{\textrm{i-th segment}} (\vecx_i( s)-\vecx
+\frac{\ell_2}{2}\,\vece_{2,\parallel})\times
\vecf_i(\vecs)\,\mbox{d}\vecs\,, \quad\quad i=1,2\,,
$$
where $\times$ stands for the cross product, the expressions of
$\vecx_i( s)$ and $\vecf_i(\vecs)$ are given by \eqref{eq_local_force}
through \eqref{eq:7} and integration takes part for $s\in[0,\ell_1]$
on the first link and $s\in[-\ell_2/2,\ell_2/2]$ on the second
link. This yields
\begin{align} \nonumber &\vecT_1^h =
\eta_1{\ell_1}^2\left(\,\vphantom{\dot\theta} {\textstyle\frac12} \,
(-\sin\alpha\,\dot x_\theta + \cos\alpha\, \dot y_\theta) \right.  \\
&\hspace*{8em} \left.  -{\textstyle\frac14} \ell_2
\cos\alpha\,\dot\theta -{\textstyle\frac13}
\eta_1\,(\dot\theta+\dot\alpha) \right) \vece_z \,, \label{eq:Torque1}
\\ &\vecT_2^h= -{\textstyle\frac12} \eta_2{\ell_2}^2 \left(\dot
y_\theta+{\textstyle\frac16}\ell_2 \,\dot \theta \right)\vece_z \,,
\label{eq:Torque2}
\end{align} where $\dot x_\theta$ and $\dot y_\theta$ are defined as:
$\left(\begin{smallmatrix}\dot{x}_\theta\\\dot{y}_\theta\end{smallmatrix}\right)
=\mathbf{R}_{-\theta} \dot\vecx$.

Substituting the elements in equations \eqref{eq:sum_of_forces} with
their just computed expressions yields
\begin{eqnarray}
\vecM^h(\theta,\alpha)
\begin{pmatrix} \dot x\\ \dot y \\ \dot \theta \\ \dot \alpha  \end{pmatrix} =
%\quad\quad\quad\quad\quad\quad\quad\quad\quad\quad \nonumber\\
\begin{pmatrix}0\\0 \\ - M_1 \left(\cos\alpha  H_{\perp}-\sin\alpha H_{\parallel}\right) - M_2 H_{\perp}\\  - \kappa \alpha - M_2 H_{\perp} \end{pmatrix} 
\label{eq:dynMat}
\end{eqnarray}
where
\begin{align}
  \label{eq:13} &\vecM^h(\theta,\alpha)= \left (\begin{array}{c|c}
\mathbf{R}_{\theta+\alpha} & 0 \\ \hline 0 & \vecI_2
\\ \end{array}\!\!\right) \mathbf{E}(\alpha) \left (\begin{array}{c|c}
\mathbf{R}_{-\theta} & 0 \\ \hline 0 & \vecI_2 \\ \end{array}\right) ,
\end{align}
with
\begin{align} & \mathbf{E}(\alpha)=\left (\!\begin{array}{c|c}
\!E_{11}(\alpha)\! &\! E_{12}(\alpha) \!\\ \hline \! E_{21}(\alpha)
\!&\! E_{22}(\alpha) \! \\
\end{array}\!\right), \label{eq:E11}
\end{align}
and
\begin{align}
  &
E_{11}= {\scriptscriptstyle
    \begin{pmatrix} -\left( \xi_{{1}}\ell_{{1}}+\xi_{{2}}\ell_{{2}}
\right)\cos\alpha & -\left(\xi_{{1}}\ell_{{1}}+\eta_{{2}}
\ell_{{2}}\right)\sin\alpha \\ \left(
\eta_{{1}}\ell_{{1}}+\xi_{{2}}\ell_{{2}} \right) \sin\alpha &- \left(
\eta_{{1}}\ell_{{1}} +\eta_{{2}} \ell_{{2}} \right) \cos\alpha
    \end{pmatrix} } \,, \nonumber\\
  &
E_{12}= {\scriptscriptstyle
    \begin{pmatrix} \frac12 \xi_{{1}}\ell_{{1}}\ell_{{2}}\sin\alpha
&0\\ \noalign{\medskip}\frac12 \eta_{{1}}\ell_{{1}} \left(
\ell_{{1}}+\ell_{{2}}\cos\alpha \right) &\frac12
\eta_{{1}}{\ell_{{1}}}^{\!2}
    \end{pmatrix} } \,, \nonumber\\
  &
E_{21}= {\textstyle\frac12} {\scriptscriptstyle
    \begin{pmatrix} \eta_1\ell_1^2&\eta_2\ell_2^2 \\[.4ex]
0&\eta_2\ell_2^2
    \end{pmatrix}
         \begin {pmatrix} - \sin\alpha &\cos\alpha \\[.5ex] 0&-1\end
{pmatrix} } \,, \nonumber\\
& 
E_{22}= - {\scriptscriptstyle \begin{pmatrix}
\eta_1\ell_1^2&\eta_2\ell_2^2 \\[.4ex] 0&\eta_2\ell_2^2
\end{pmatrix}\begin{pmatrix} \frac14 \,\ell_{{2}}\cos\alpha
+\frac13\,\ell_{{1}}&\frac13\,\ell_{{1}}\\
\frac1{12}\,\ell_{{2}}&0 \end{pmatrix} } \,.\nonumber
\end{align}
The determinant of $E(\alpha)$ is given by
\begin{align}
  \nonumber
  &\textstyle
    -\frac19\eta_1\eta_2\,\ell_1^{\,3}\ell_2^{\,3} \left(
    \frac14\,(\xi_1\ell_1+\xi_2\ell_2)\,(\eta_1 \ell_1+\eta_2\ell_2)\cos^{\!2}\alpha 
    \right.\\
  &\left.\textstyle \hspace*{6em}
    +\bigl(\xi_1\ell_1+\frac14\,\eta_2\ell_2\bigr)\bigl(\frac14\,\eta_1\ell_1 +\xi_2\ell_2\bigr)\sin^{\!2}\alpha
    \right),\nonumber
\end{align}
hence it remains negative, $E(\alpha)$ is invertible for all $\alpha$,
so is $M_h$, and the dynamics \eqref{eq:dynMat} of the swimmer can be
written as a control system
\begin{equation}
  \label{eq:dynamics}
  \dot \vecz = \vecF_0(\vecz) + H_{\parallel}\,\vecF_1(\vecz) +
  H_{\perp}\,\vecF_2(\vecz) \ \ \ \text{with}\ 
  \vecz=\begin{pmatrix} x\\y\\ \theta \\ \alpha \end{pmatrix} \,,
\end{equation}
affine with respect to the controls $H_{\parallel}$ and $H_{\perp}$
where $\vecF_0$, $\vecF_1$, $\vecF_2$ are vector fields on $\R^2\times
S^1\times S^1$ expressed as follows.
\begin{proposition}
  \label{prop:F0F1F2}
The vector fields $\vecF_0$, $\vecF_1$, $\vecF_2$ of system
\eqref{eq:dynamics} are given by
\begin{equation}
  \label{eq:12}
  \begin{split} \vecF_0(\vecz) = \kappa \alpha \vecX_4\,,\quad
\vecF_1(\vecz) = M_1 \sin\alpha \vecX_3\,,\,\\ \quad \vecF_2(\vecz)=
-\left(M_1 \cos\alpha +M_2\right)\vecX_3 -M_2 \vecX_4\,,
  \end{split}
\end{equation}
where $\vecX_3$ and $\vecX_4$ are the vector fields whose vector of
coordinates in the ``rotating'' basis
$$
(\cos\theta\frac\partial{\partial{x}}+\sin\theta\frac\partial{\partial{y}},-\sin\theta\frac\partial{\partial{x}}+\cos\theta\frac\partial{\partial{y}},\frac\partial{\partial{\theta}},\frac\partial{\partial{\alpha}})
$$
are respectively the third and fourth columns of $E(\alpha)^{-1} $.
\end{proposition}
\begin{proof}
  This is easily derived from \eqref{eq:dynMat} and
\eqref{eq:13}. Note that the matrix depending on $\theta+\alpha$ in
\eqref{eq:13} plays no role in \eqref{eq:dynMat} because it leaves the
right-hand side invariant.
\end{proof}

% =====================================================================================
% =====================================================================================
% =====================================================================================
\section{Controllability result}
\label{sec:MainResult}

According to \eqref{eq:12} and the fact that the vector field
$\vecX_4$ does not vanish (its components are one column of an
invertible matrix), the zeroes of $\vecF_0$ are exactly described by
$(x,y,\theta)\in \R^2\times S^1$ arbitrary and $\alpha=0$. Hence these
are the equilibrium positions of system \eqref{eq:dynamics} when the
control, namely the magnetic field, is zero:
$H_{\parallel}=H_{\perp}=0$.  Local controllability describes how all
points $(x,y,\theta,\alpha)$ sufficiently close to a fixed equilibrium
$(x^e,y^e,\theta^e,0)$ can be reached by applying \emph{small}
magnetic fields for a \emph{small duration} using a trajectory that
\emph{remains close to} $(x^e,y^e,\theta^e,0)$.  The local
controllability result stated below does not ensure that the control
can be chosen arbitrarily small, i.e. this result is not local with
respect to the control; it however gives an explicit bound on the
needed control.

\subsection{Main result}

Let us first point two cases where controllability cannot hold. In the
first case, the variables $\alpha$ and $\theta$ may be controlled but
$x(t)$ and $y(t)$ are related to $\alpha(t)$ and $\theta(t)$ by a
formula, valid everywhere, that does not depend on the control: the
system is nowhere controllable.  In the second case, the variables
$x,y,\theta$ may be controlled but $\alpha(0)=0$ implies $\alpha(t)=0$
for all $t$, irrespective of the control, thus forbidding local
controllability around any $(x^e,y^e,\theta^e,0)$ but possibly not
away from $\{\alpha=0\}$.
\begin{proposition}
  \label{thm:NonCommandable1}
  \begin{itemize}
  \item If $\eta_1-\xi_1=\eta_2-\xi_2=0$, then there exists two
    constants $x^0,y^0$, that depend on the initial conditions of the
    state but not on the control (the magnetic field $H$), such that
    solutions of \eqref{eq:dynamics} satisfy, for all $t$,
    \begin{equation*}
    \begin{split}
        x(t)=x^0+
        \frac{\eta_1\ell_1\left(\ell_1\cos(\alpha(t)+\theta(t))+\ell_2\cos\alpha(t)\right)}
        {2(\eta_1\ell_1+\eta_2\ell_2)}
        \,,\\
        y(t)=y^0+
        \frac{\eta_1\ell_1\left(\ell_1\sin(\alpha(t)+\theta(t))+\ell_2\sin\alpha(t)\right)}
        {2(\eta_1\ell_1+\eta_2\ell_2)}\,.\
   \end{split}
   \end{equation*}
   \item If
    \begin{equation}   
      \left(
      3+4\frac{\ell_2}{\ell_1}+\frac{\eta_2{\ell_2}^{2}}{\eta_1{\ell_1}^{2}}
      \right) M_1
      -  
      \left( 
      3+4\frac{\ell_1}{\ell_2}+\frac{\eta_1{\ell_1}^{2}}{\eta_2{\ell_2}^{2}}
      \right) M_2
      =0\,,\quad \label{eq:3}
    \end{equation}
then the set $\{\alpha=0\}$ is invariant for equations \eqref{eq:dynamics}.
\end{itemize}
\end{proposition}
\begin{proof}
A simple computation shows that if $\eta_1-\xi_1=\eta_2-\xi_2=0$, then the time-derivatives of
\begin{equation*}
    x - \frac
{\eta_1\ell_1\left(\ell_1\cos(\alpha+\theta)+\ell_2\cos\alpha\right)}
{2(\eta_1\ell_1+\eta_2\ell_2)}
\end{equation*}
  and
\begin{equation*}
    y - \frac
{\eta_1\ell_1\left(\ell_1\sin(\alpha+\theta)+\ell_2\sin\alpha\right)}
{2(\eta_1\ell_1+\eta_2\ell_2)}
\end{equation*}
  are zero (i.e. these are first intergrals of the system), hence the
  first point.  Another computation (see details further, a few lines
  after \eqref{eq:21})
  shows that $\dot\alpha$ is zero when \eqref{eq:3} is satisfied and
  $\alpha=0$, hence proving the second point.
\end{proof}

From now on, we make the following assumption:
\begin{hypo}
  \label{ass}
  The constants $\ell_1$, $\ell_2$, $\xi_1$, $\xi_2$, $\eta_1$,
  $\eta_2$, $M_1$, $M_2$, $\kappa$ characterizing the system are such
  that $\ell_1$, $\ell_2$, $\xi_1$, $\xi_2$, $\eta_1$, $\eta_2$ and
  $\kappa$ are positive, $M_1$ and $M_2$ are nonzero, and
  \begin{align}
    (\,\eta_1-\xi_1\,,\,\eta_2-\xi_2\,)\,\neq\,(0,0)\,,  \label{eq:9} \\
    \eta_1\geq\xi_1\,,\ \ \eta_2\geq\xi_2\,,
    \label{eq:11} \\    
    \left(
    3+4\frac{\ell_2}{\ell_1}+\frac{\eta_2{\ell_2}^{2}}{\eta_1{\ell_1}^{2}} \right) 
    M_1 %\quad\quad\quad\quad\quad\quad\quad\quad\quad\quad\nonumber\\ 
    -
    \left(
    3+4\frac{\ell_1}{\ell_2}
    +\frac {\eta_1{\ell_1}^{2}}{\eta_2{\ell_2}^{2}}  \right) 
    M_2
    \neq0\,.\label{eq:8}
  \end{align}
\end{hypo}
\begin{remark}
  \label{rem:ass}
  Conditions \eqref{eq:9} and \eqref{eq:8} exactly exclude the cases
  in which Proposition~\ref{thm:NonCommandable1} applies.  We have
  added conditions \eqref{eq:11}.  Physically, these inequalities are
  usually satisfied: they state that the normal drag force is more
  important than the tangential one.  Technically, they avoid numerous
  sub-cases. For instance, \eqref{eq:5} is not satisfied if
  $\eta_1 l_1 ( \eta_2-\xi_2) +\eta_2 l_2 ( \eta_1-\xi_1) =0$, that
  does not contradict \eqref{eq:9} but contradicts \eqref{eq:9} and
  \eqref{eq:8}; the same happens in the proof of
  Lemma~\ref{lem:rang}. Theorem \ref{Thm:main} however still holds
  without \eqref{eq:11}, that we added to make the proofs lighter.
\end{remark}
Let us now state our main result.
\begin{theorem}[Local controllability]
    \label{Thm:main}
    Let Assumption~\ref{ass} hold.  
    Fix an equilibrium $\vecz^e=(x^e,y^e,\theta^e,0)$.
    Let $\mathcal{W}$ be a neighborhood of $\vecz^e$ in 
    $\R^2\times S^1\times S^1$ and $T,\varepsilon$ positive numbers,
    then there exists another neighborhood $\mathcal{V}\subset \mathcal{W}$ of $\vecz^e$ 
    such that, for any $\vecz^i=(\vecx^i,\theta^i,\alpha^i)$ and
    $\vecz ^f=(\vecx^f,\theta^f,\alpha^f)$ in $\mathcal{V}$, 
    there exist bounded measurable functions $H_{\parallel}$ and
    $H_{\perp}$ in $L^\infty([0,T]),I\!\!R)$ such that
    \begin{equation}
      \label{eq:44}
      \|H_{\perp}\|_\infty<\varepsilon\,,\ \ \ \|H_{\parallel}\|_\infty<2\,\kappa\left|\frac{M_2+M_1}{M_2\,M_1}\right|+\varepsilon
    \end{equation}
    and, if $t\mapsto \vecz(t)=(\vecx(t),\theta(t),\alpha(t))$ is the
    solution of \eqref{eq:dynamics} starting at $\vecz^i$, then
    $\vecz(T)=\vecz^f$ and $\vecz(t)\in\mathcal{W}$ for all time
    $t\in[0,T]$.
\end{theorem}
We postpone the proof of this result to make further comments.
They are made simpler by restricting to one special equilibrium:
\begin{proposition}[Invariance]
    \label{prop-invar}
    If Theorem~\ref{Thm:main} (resp. any kind of local controllability
    like STLC, see Definition~\ref{def:stlc} below) holds in the special
    case where $\vecz^e$ is ``the origin''
    \begin{equation}
      \label{eq:Me}
      \mathbf{O}=(0,0,0,0)\in \R^2\times S^1\times S^1\,,
    \end{equation}
    i.e. when $x^e=y^e=\theta^e=0$, then the same holds for arbitrary $\vecz^e=(x^e,y^e,\theta^e,0)$.
\end{proposition}
\begin{proof}
  Solutions of \eqref{eq:dynamics} are invariant under the transformations
  \begin{displaymath}
    \left( \left(\begin{smallmatrix}x\\y\end{smallmatrix}\right) ,\theta,\alpha,H_{\parallel},H_{\perp} \right) 
    \mapsto 
    \left(
       \mathbf{R}_{\bar{\theta}}
       \left(\begin{smallmatrix}x+\bar x\\y+\bar y\end{smallmatrix}\right)
       ,\theta+\bar\theta,\alpha,H_{\parallel},H_{\perp}
     \right)\,,
  \end{displaymath}
   hence everything may be carried from a neighborhood of $\mathbf{O}$
   to a neighborhood of an arbitrary $(x^e,y^e,\theta^e,0)$.
\end{proof}

\subsection{Discussion}
\label{sec-discussion}
We assume $(x^e,y^e,\theta^e,0)=(0,0,0,0)=\mathbf{O}$ without loss of
generality.  The strongest notion of local controllability is the
following (definition taken from \cite[Definition~3.2]{Coron56}).  It
is interesting and natural in that that it is local both in control
and in space.
\begin{definition}[STLC]
\label{def:stlc}
The system \eqref{eq:dynamics} is said to be \emph{small time locally
  controllable} (STLC) at equilibrium $\vecz^e=(x^e,y^e,\theta^e,0)$
and control $(0,0)$ if and only if, for any neighborhood $\mathcal{W}$
of $\vecz^e$, any $T>0$, and any $\varepsilon>0$, there exists another
neighborhood $\mathcal{V}\subset \mathcal{W}$ of $\vecz^e$ such that,
for any $\vecz^i=(\vecx^i,\theta^i,\alpha^i)$ and
$\vecz^f=(\vecx^f,\theta^f,\alpha^f)$ in $\mathcal{V}$, there exist
bounded measurable functions $H_{\parallel}$ and $H_{\perp}$ in
$L^\infty([0,T]),I\!\!R)$ such that
$\|H_{\perp}\|_\infty<\varepsilon$,
$\|H_{\parallel}\|_\infty<\varepsilon$, and, if $t\mapsto
\vecz(t)=(\vecx(t),\theta(t),\alpha(t))$ is the solution of
\eqref{eq:dynamics} starting at $\vecz^i$, then $\vecz(T)=\vecz^f$ and
$\vecz(t)\in\mathcal{W}$ for all time $t\in[0,T]$.
\end{definition}

Theorem~\ref{Thm:main} establishes a rather strong form of local
controllability with an \emph{explicit} bound on the controls. It is
however weaker than STLC because the bound on the controls remains
larger than
$2\,\kappa\,\frac{\left|M_1+M_2\right|}{\left|M_2\,M_1\right|}$, hence
it does not go to zero when $\mathcal{V}$ becomes smaller and smaller.
Theorem \ref{Thm:main} establishes STLC only if $M_1+M_2=0$, and we do
not know whether STLC holds or not when $M_1+M_2\neq0$ (physically, we
expect to have some sort of lower-bound on the magnetic field to
deform the swimmer, but this is not formalized).

Let us review the classical ways to establish STLC, namely the
``linear test'' and the theorem on ``bad and good brackets'' due to
H. Sussmann \cite{Suss87siam}, and explain why they fail.  The notion
of linearized system along a trajectory is instrumental.
\begin{definition}
The linearized control system of \eqref{eq:dynamics} around a
trajectory $t\mapsto(\vecz^*(t)$, $H_\perp^*(t),H_{\parallel}^*(t))$
defined on $[0,T]$ for some $T>0$ is the time-varying control system
\begin{equation}
\label{eq:linear_system}
\dot{\vecy} = \vecA(t) \vecy + \vecB(t) \vecv\,,
\end{equation}
where $\vecA(t)$ is the Jacobian of $\vecz\mapsto
\vecF_0(\vecz)+H^*_{\perp}(t)\vecF_2(\vecz)+H^*_{\parallel}(t)\vecF_1(\vecz)$
with respect to $\vecz$ at $\vecz=\vecz^*(t)$ and the two columns of
$\vecB(t)$ are $\vecF_1(\vecz^*(t))$ and $\vecF_2(\vecz^*(t))$.
\end{definition}

If the trajectory is simply the equilibrium point $\mathbf{O}$ and the
reference controls are zero, $\vecA$ and $\vecB$ do not depend on time
and $\vecA$ is simply the Jacobian of $\vecz\mapsto\vecF_0(\vecz)$ at
$\vecz=\mathbf{O}$.
\begin{proposition}
  The controllability matrix
$\mathcal{C}=\left[\vecB,\vecA\vecB\vphantom{\vecA^2\vecB},\vecA^2\vecB,
\vecA^3\vecB\right]$ (8 columns, 4 lines) for the linearized control
system of \eqref{eq:dynamics} at the equilibrium $\mathbf{O}$ has rank
at most 2.
\end{proposition}
\begin{proof}
  One deduces from \eqref{eq:12} that on the one hand the Jacobian of
$\vecz\mapsto\vecF_0(\vecz)$ at $\mathbf{O}$ has only one nonzero
column, proportional to $\vecX_4(\mathbf{O})$, and on the other hand the
vector field $\vecF_1$ is zero at $\mathbf{O}$. Hence the rank of
$\mathcal{C}$ is at most the rank of
$\{\vecX_4(\mathbf{O}),\vecF_2(\mathbf{O})\}$.
\end{proof}
The linear test for local controllability \cite[Theorem 3.8]{Coron56}
states that a nonlinear control system is STLC at an equilibrium if
its linearized control system at this point is controllable.
According to the Kalman rank condition \cite[Theorem 1.16]{Coron56},
the latter linearized system is not controllable because the rank of
$\mathcal{C}$ is stricly less than 4.
Hence, the linear test cannot be applied.

A more general sufficient condition was introduced in \cite[section
7]{Suss87siam} and recalled in \cite[section 3.4]{Coron56}. It
requires the following notions.
\begin{definition}[LARC]
  System \eqref{eq:dynamics} satisfies the LARC (Lie Algebra Rank
Condition) at $\mathbf{O}$ if and only if the values at $\mathbf{O}$
of all iterated Lie brackets of the vector fields $\vecF_0$, $\vecF_1$
and $\vecF_2$ span a vector space of dimension 4.
\end{definition}
\noindent
Now, for $\theta,\eta$ positive numbers and $h$ an iterated Lie
bracket of the vector fields $\vecF_0$, $\vecF_1$, $\vecF_2$, let
\begin{itemize}
\item $\sigma(h)$ be the sum of $h$ and the iterated Lie bracket
  obtained by exchanging $\vecF_1$ and $\vecF_2$ in $h$,
\item $\delta_i(h)\in\N$ ($i=0,1,2$) be the number of times the vector
  field $\vecF_i$ appears in $h$,
\item $\rho_\theta(h)$ be given by
  $\rho_\theta(h)=\theta\delta_0(h)+\delta_1(h)+\delta_2(h)$,
\item $G_\eta$ be the vector subspace spanned by all vectors
  $g(\mathbf{O})$ where $g$ is an iterated bracket of
  $\vecF_0,\vecF_1,\vecF_2$ such that $\rho_\theta(g)<\eta$.
\end{itemize}
\textit{Note:} the LARC is equivalent to $G_\eta$ having dimension 4
for large enough $\eta$.
\begin{definition}[Sussman's condition $S(\theta)$ \cite{Suss87siam}]
  Let $\theta$ be a number, $0\leq\theta\leq1$. System
  \eqref{eq:dynamics} satisfies the condition $S(\theta)$ at
  $\mathbf{O}$ if and only if it satisfies the LARC and any iterated
  Lie bracket $h$ of the vector fields $\vecF_0$, $\vecF_1$, $\vecF_2$
  such that $\delta_0(h)$ is odd and both $\delta_1(h)$ and
  $\delta_2(h)$ are even (``bad'' brackets) satisfies
  $\sigma(h)(\mathbf{O})\in G_{\rho_\theta(h)}$.
\end{definition}

\noindent
The main theorem in \cite{Suss87siam} states that system
\eqref{eq:dynamics} is STLC if the condition $S(\theta)$ holds for at
least one $\theta$ in $[0,1]$. Proposition \ref{prop:Suss} below shows
that this sufficient condition cannot be applied, except if
$M_1+M_2=0$.  This is consistant with our Theorem~\ref{Thm:main}, that
establishes STLC only in this case.
\begin{proposition}
\label{prop:Suss}
Assume that the parameters of the system \eqref{eq:dynamics} satisfy
Assumption~\ref{ass}. Then
\begin{enumerate}
\item the LARC is satisfied at $\mathbf{O}$,
\item if $M_1+M_2\neq0$, then $S(\theta)$ is not satisfied at
$\mathbf{O}$ for any $\theta$ in $[0,1]$,
\item if $M_1+M_2=0$, then $S(1)$ is satisfied at $\mathbf{O}$.
\end{enumerate}
\end{proposition}
\begin{proof}
In order to save space, we denote by $(\cdots)$ any coefficient
whose value does not matter and by $f_{i_1 i_2 .. i_m}$ or $X_{i_1 i_2
.. i_m}$ the value at $\mathbf{O}$ of the iterated Lie bracket $
[\vecF_{i_1},[\vecF_{i_2},\cdots \vecF_{i_m}]\cdots]]$ or
$[\vecX_{i_1},[\vecX_{i_2},\cdots \vecX_{i_m}]\cdots]]$: for example,
$\displaystyle
f_0=\vecF_0(\mathbf{O}),\ 
X_{34}=[\vecX_3,\vecX_4](\mathbf{O}),\ %\nonumber\\
f_{1021}=[\vecF_1 ,[\vecF_0 ,[\vecF_2 ,\vecF_1]]](\mathbf{O})\,.\nonumber
$
% \begin{eqnarray}
% f_0=\vecF_0(\mathbf{O}),\ 
% X_{34}=[\vecX_3,\vecX_4](\mathbf{O}),\ %\nonumber\\
% f_{1021}=[\vecF_1 ,[\vecF_0 ,[\vecF_2 ,\vecF_1]]](\mathbf{O})\,.\nonumber
% \end{eqnarray}

\smallskip

Computing Lie brackets with a computer algebra software (Maple),
taking their value at $\mathbf{O} $ and forming determinants, we show
that
$$\det\left(X_3,X_4,X_{34},X_{334}\right)
\ \ \text{and}\ \ \det\left(X_3,X_4,X_{34},X_{434}\right)$$
cannot be both zero if \eqref{eq:9} and \eqref{eq:11} hold. This
proves:
\begin{equation}
  \label{eq:5}
  \Rank\{X_3,X_4,X_{34},X_{334},X_{434}\}=4\,.
\end{equation}
Point (1) follows because, with $L$ a function of the constants that is
nonzero if and only if \eqref{eq:8} is met, one has
\begin{eqnarray}
  f_{02}=\kappa L \,X_4\,,\ 
f_{12}=M_1 L\, X_3\,,\ %\nonumber \\
f_{212}=2L \,X_{34}+(\cdots)X_3\,,\label{eq:10}
\end{eqnarray}
and, modulo a linear combination of $X_3$, $X_4$ and $X_{34}$,
\begin{eqnarray}
  f_{12212}-f_{21212}=2M_1L^2\,X_{334}\,,\ %\nonumber \\
  f_{02212}-f_{20212}=2\kappa L^2\,X_{434}\,.\label{eq:17}
\end{eqnarray}
These are obtained from \eqref{eq:12}, the expressions of $\vecX_3$
and $\vecX_4$ are needed to compute the number $L$.

% To prove point (2), we use the ``bad'' brackets
% $[\vecF_2,[\vecF_0,\vecF_2]]$.  We computed the values $f_{101}$ and
% $f_{202}$ at $\mathbf{O}$ and, since $f_{101}=0$, one has values at
% zero are
% \begin{equation}
%   \label{eq:28}
%   \sigma(f_{202}) =f_{202}=-2\kappa(M_1+M_2)L\,X_{34}+(\cdots)X_4\,.
% \end{equation}
% One has $\rho_\theta(f_{202})=2+\theta$ and $G_{2+\theta}$ is the
% vector space spanned by $X_3$ and $X_4$ because it is, by definition,
% spanned by $f_0$, $f_1$, $f_2$, $f_{01}$, $f_{02}$, $f_{12}$ (that
% span the same space as $X_3$ and $X_4$ according to \eqref{eq:10} and
% the fact that $f_2$ is a linear combination of $X_3$ and $X_4$ while
% $f_0=f_1=f_{01}=0$) and, depending on the value of $\theta$, by some
% $f_{0\cdots01}$ and $f_{0\cdots02}$ that are all colinear to $X_4$
% (see \eqref{eq:12}). This implies that if $M_1+M_2\neq0$, then
% $\sigma(f_{202})\notin G_{2+\theta}$, hence the condition $S(\theta)$
% does not hold.
To prove point (2), we use the ``bad'' bracket $h=[\mathbf{F}_2,[\mathbf{F}_0,\mathbf{F}_2]]$.
Since $f_{101}=0$, one has $\sigma(h)(\mathbf{O})=f_{202}$. 
Computing $f_{202}$ yields:
\begin{equation}
  \label{eq:28}
  \sigma(h)(\mathbf{O})=-2\kappa(M_1+M_2)L\,X_{34}+(\cdots)X_4\,.
\end{equation}
One has $\rho_\theta(h)=2+\theta$. 
$G_{2+\theta}$ is, by definition, the vector space spanned 
by $f_0$, $f_1$, $f_2$, $f_{01}$, $f_{02}$, $f_{12}$, that are linear combinations of $X_3,X_4$
      (see (16) and (25)), 
%     (see \eqref{eq:12} and \eqref{eq:10}), 
and, depending on the value of $\theta$,
   by some $f_{0\cdots01}$, that are all zero, 
   and by some $f_{0\cdots02}$, that are all colinear to $X_4$. 
Hence $G_{2+\theta}$ is spanned by $X_3$ and $X_4$.  
Considering equation \eqref{eq:28} where $M_1+M_2\neq0$ is assumed, 
one then has $\sigma(h)(\mathbf{O})\notin G_{2+\theta}$, 
proving that the condition $S(\theta)$ does not hold.

For point (3), we assume $M_1+M_2=0$ and take $\theta=1$ so that
$\rho_\theta(h)$ is the order of the iterated Lie bracket
$h$. According to \eqref{eq:10} and \eqref{eq:17}, Lie brackets of
order 5 generate the whole space, i.e. $G_\eta$ is the whole tangent
space $\R^4$ if $\eta>5$. Besides $[\mathbf{F}_1,[\mathbf{F}_0,\mathbf{F}_1]]$ and $[\mathbf{F}_2,[\mathbf{F}_0,\mathbf{F}_2]]$, the bad Lie
brackets of order at most 5 are these that contain three times
$\vecF_0$ and two times either $\vecF_1$ or $\vecF_2$, these that
contain one time $\vecF_0$, two times $\vecF_1$ and two times
$\vecF_2$, and these that contain one time $\vecF_0$ and four times
either $\vecF_1$ or $\vecF_2$; it can be checked that they all belong
to $G_5$, spanned by $X_3$, $X_4$ $X_{34}$ and $X_{434}$.
\end{proof}

\subsection{Proof of Theorem~\ref{Thm:main}}
\label{sec-proof}

This proof relies on the \textit{return
method}, introduced by J.-M. Coron in \cite{Coron92}
for stabilization purposes, and exposed in
\cite[chapter 6]{Coron56}. It has mostly been used to establish controllability
results for infinite dimensional control systems (PDEs).
The idea of the method is to find a trajectory (``loop'') of system
\eqref{eq:dynamics} such that it starts and ends at the equilibrium
$\mathbf{O}$ and the linearized control system around this trajectory
is controllable, and then conclude by using the implicit function
theorem that one can go from any state close to the equilibrium to any
other final state close to the equilibrium.
The proof relies on Lemma \ref{lem:return}, that identifies a family
of bounded controls producing ``loops'' from $\mathbf{O}$ to
$\mathbf{O}$ and on Lemma \ref{lem:linearise} that shows 
controllability of the linearized system \eqref{eq:dynamics} around
some of these loop trajectories.
\begin{lemma}[return trajectory]
  \label{lem:return}
Let Assumption~\ref{ass} hold.  There exist positive numbers $k$,
$\overline{T}$ and $\overline{H}$ with the following property: for any
$T$, $0<T\leq\overline{T}$, and any measurable control $t \mapsto
H(t)=(H_{\perp}(t),H_{\parallel}(t))$ defined on $[0,T/2]$ and bounded
by $\overline{H}$, there is a bounded measurable control $t\mapsto
H^*(t)=(H^*_{\perp}(t),H^*_{\parallel}(t))$ defined on $[0,T]$ such
that
  \begin{align}
    \label{eq:23}
    &H^*_{\perp}(t)=H_{\perp}(t) \ \textrm{and} \  H^*_{\parallel}(t) =H_{\parallel}(t)\,,
      \quad\ 0\leq t\leq\frac{T}2\,,\\
    \label{eq:24}
    &\|H_\perp^*(.)\|_\infty\leq
      k\,\|H(.)\|_\infty\,,
      % (1+ k\,T)\,\|H(.)\|_\infty\,,
    \\ 
    \label{eq:25}
    &\|H_{\parallel}^*(.)\|_\infty\leq
      2\kappa\left|\frac{M_1+M_2}{M_1\,M_2}\right| +
      k\,\|H(.)\|_\infty\,,
      % (1+k\,T)\,\|H(.)\|_\infty\,,
  \end{align} 
and, if $t\mapsto \vecz^*(t)= (\vecx^*(t),\theta^*(t),\alpha^*(t))$ is
the solution of system \eqref{eq:dynamics} with control $H^*$ and
initial condition $\vecz^*(0)=\mathbf{O}$, then
\begin{equation}
  \label{eq:26}
  \vecz^*(t)=\vecz^*(T-t)\,,\ \ 0\leq t\leq T\,,
\end{equation}
and $\alpha^*(t)$ remains in $[-\frac{\pi}{2},\frac{\pi}{2}]$, for all t.
\end{lemma}
\begin{lemma}[linear controllability along trajectories]
  \label{lem:linearise}
For any number $\beta$, denote by $t\mapsto \vecz^{\beta}(t)$ the
solution of \eqref{eq:dynamics} with initial condition
$\vecz^{\beta}(0)=\mathbf{O}$ and (constant) controls
\begin{equation}
  \label{eq:15}
  H^{\beta}_\perp(t)=\beta\,,\ \ 
  H^{\beta}_\parallel(t)=0\,.
\end{equation}
It is defined on $[0,+\infty)$.  Under Assumption~\ref{ass}, there
exist arbitrarily small positive values of $\beta$ such that the
linearized system \eqref{eq:linear_system} of \eqref{eq:dynamics}
around $(\vecz^{\beta}(.),$ $H^{\beta}_\perp(.),
H^{\beta}_\parallel(.))$ is controllable on $[0,\tau]$ for any
positive $\tau$.
\end{lemma}
% ===============================================================================

These two lemmas are proved later. Let us first use them.
\begin{proof}[Proof of Theorem~\ref{Thm:main}]
Let us first prove the theorem assuming $\vecz^i=\mathbf{O}$.  Let
$\varepsilon>0$, $T>0$, the equilibrium $\mathbf{O}$ and its
neighborhood $\mathcal{W}$ be given.  Lemma~\ref{lem:return} provides
some constants $k$, $\overline{H}$, $\overline{T}$ that depend only on
$\mathbf{O}$ and the constants of the problem.

In the sequel we assume $T\leq \overline{T}$ without loss of
generality; indeed, if $T>\overline{T}$, first solve the problem for
$T=\overline{T}$ (and the same $\varepsilon$, $\mathbf{O}$) and denote
by $t\mapsto
H^{\overline{T}}(t)=(H_{\parallel}^{\overline{T}}(t),H_{\perp}^{\overline{T}}(t))$
the control solving that problem, then the solution for the actual
$T>\overline{T}$ is given by
$$H(t)=
\begin{cases}
  \ 0& \text{if }0\leq t<T-\overline{T}\,, \\
  H^{\overline{T}}(t-T+\overline{T})&\text{if } T-\overline{T}\leq
  t\leq T\,.
\end{cases}
$$

From Lemma~\ref{lem:linearise}, there exists some $\beta>0$ such that
$\beta<\varepsilon/2k$, $\beta<\overline{H}$, and the linearized
system \eqref{eq:linear_system} of \eqref{eq:dynamics} around
$\vecz^{\beta}(.)=(\vecx^{\beta}(.),$
$H^{\beta}_\perp(.), H^{\beta}_\parallel(.))$ (defined in the lemma)
is controllable on $[0,\tau]$ for any positive $\tau$.  Since
$\vecz^{\beta}(0)=\mathbf{O}$, there is some
$\overline{\overline{T}}>0$ such that
$\vecz^*([0,\overline{\overline{T}}])\subset\mathcal{W}$.  We assume
that $T\leq\overline{\overline{T}}$ without loss of generality for the
same reason that allowed us to assume $T\leq\overline{T}$ above.

Now apply Lemma~\ref{lem:return} with
\begin{eqnarray}
  H_{\perp}(t)=H^{\beta}_{\perp}(t)=\beta\,,\quad\quad\quad\quad\quad\nonumber\\
  H_{\parallel}(t)=H^{\beta}(t)=0\,,\ \
  \ 0\leq t\leq\frac T 2\ . \label{eq:14}
\end{eqnarray}
% with $\beta^*$ given by Lemma~\ref{lem:linearise}, and apply
% Lemma~\ref{lem:return} (this requires $T\leq\overline{T}$):
This yields % , if $T<1/k$,
a control $t\mapsto(H^{*}_{\perp}(t),H^{*}_{\parallel}(t))$ associated
with a solution $t\mapsto\vecz^*(t)$ of \eqref{eq:dynamics}, both
defined on $[0,T]$, such that
\begin{equation}
  \label{eq:16}
  \begin{aligned}
    \|H^{*}_{\perp}\|_\infty<\frac\varepsilon2\,,\
    \|H^{*}_{\parallel}\|_\infty<
    \frac{2\kappa\left|M_1+M_2\right|}{M_1\,M_2}+ \frac\varepsilon2\,,
    \\
    \vecz^*(0)=\vecz^*(T)=\mathbf{O}\,,\quad
    \,\vecz^*(t)\in\mathcal{W}\,,\ t\in[0,T]\,,
  \end{aligned}
\end{equation}
and the linear approximation of \eqref{eq:dynamics} along this
solution is controllable.  Note that $\vecz^*(t)$ remains in
$\mathcal{W}$ because $T\leq\overline{\overline{T}}$.

Let the end-point mapping $\mathcal{E}:L^\infty([0,T],\R^2)\to\R^4$ be
the one that maps a control on $[0,T]$ to the point $\vecz(T)$ with
$\vecz(.)$ the solution of the system \eqref{eq:dynamics} with this
control and initial value $\vecz(0)=\mathbf{O}$. We have
$\mathcal{E}(H^*(.))=\mathbf{O}$ and linear controllability amounts to
$\mathcal{E}$ being a submersion at this point; hence $\mathcal{E}$
sends any neighborhood of $H^*$ in $L^\infty([0,T],\R^2)$ to a
neighborhood of $\mathbf{O}$; this yields all the properties of the
theorem restricted to $\vecz^i=\mathbf{O}$.

To obtain the theorem for arbitrary $\vecz^i,\vecz^f$, apply twice the
restricted theorem we just proved: once with $\vecz^i,\vecz^f$
replaced by $\mathbf{O},\vecz^f$, and once with $\vecz^i,\vecz^f$
replaced by $\mathbf{O},\vecz^i$ and $T<0$ (the proof can be adapted
\textit{mutatis mutandis} to $T<0$); then concatenate the two controls
in order to go from $\vecz^i$ to $\mathbf{O}$ first and then from
$\mathbf{O}$ to $\vecx^f$.
\end{proof}
\begin{proof}[Proof of Lemma~\ref{lem:return}]
Taking $\overline{T}$ and $\overline{H}$ small enough, any solution
$\vecz$ of \eqref{eq:dynamics} with $\vecz(0) = \mathbf{O}$ with a
control $H$ bounded by $\overline{H}$ satisfies $\alpha(t) \in
[-\frac{\pi}{2},\frac{\pi}{2}]$ for all $t \in [0,\overline{T}]$.
Consider $T$ positive smaller than $\overline{T}$ and an arbitrary
control $H: t\mapsto(H_{\perp}(t),H_{\parallel}(t))$ defined on
$[0,T/2]$ bounded by $\overline{H}$.  Let
$t\mapsto\vecz^*(t)=(\vecx^*(t) ,\theta^*(t),\alpha^*(t))$ be the
solution of system \eqref{eq:dynamics} on $[0,\frac{T}{2}]$ associated
with the control $H^*$ and starting at $\mathbf{O}$.  We now consider
a piecewise function $H_{\perp}^*(.)$ and $H_{\parallel}^*(.)$ defined
as
  \begin{equation}
    \label{eq:control_turn_back}
    \!
    % \left\{\!\!
    \begin{array}{ll}
      H_{\perp}^*(t) = H_{\perp}(t)\,, \quad H_{\parallel}^*(t) = H_{\parallel}(t)\,, \quad  t \in [0,\frac{T}{2}]\,,\\
      H_{\perp}^*(t) =\frac{2\,\kappa}{M_2}\alpha^*(t) - H_{\perp}(T-t)\,, \quad\,\,  t\in \left[\frac{T}{2}, T\right]\,,\\
      H_{\parallel}^*(t) = \frac{2\,\kappa\,\alpha^*(t)}{M_2\,M_1\,\sin\alpha^*(t)}\left(M_2+M_1\cos\alpha^*(t)\right) % \\
      % \quad\quad\quad\quad\quad\quad\quad\quad\quad
      - H_{\parallel}(T-t)\,, \   t\in [\frac{T}{2}, T].\\
    \end{array}
    % \right.
  \end{equation}
This particular control leads the swimmer to come back to the starting
point $\mathbf{O}$ at time $T$ by using the same path.  Indeed, the
latter condition reads
\begin{equation}
\label{eq:return}
\vecz^*(t) = \vecz^*(T-t)\,, \quad  t \in \left[\frac{T}{2}, T\right]\,.
\end{equation} 
Differentiating the latter with respect to time, we get:
\begin{eqnarray}
2\,\vecF_0(\vecz^*) + \vecF_1(\vecz^*) \left(H^*_{\parallel}+H_{\parallel}\right) 
    % \quad\quad\quad\quad\quad\quad\quad\quad\nonumber\\
    + \vecF_2(\vecz^*) \left(H^*_{\perp}+H_{\perp}\right) = 0\,, \ \;  t \in \left[\frac{T}{2}, T\right].\label{eq:system_return}
\end{eqnarray}
By using the expression \eqref{eq:12} for $\vecF_0$, $\vecF_1$ and
$\vecF_2$ and because $\vecX_3$ and $\vecX_4$ are linearly
independent, the projection onto the vector space generated by
$\vecX_i$, $i=3,4$ of the previous equality has to vanish.  The two
controls $H^*_{\parallel}$ and $H^*_{\perp}$ defined in
\eqref{eq:control_turn_back} solve the latter system of linear
equation \eqref{eq:system_return}.  Note that the latter control
functions is well-defined since $\alpha^*(t)$ remains in
$[-\frac{\pi}{2},\frac{\pi}{2}]$.

Moreover, by Gronwall Lemma, there exists a constant $k'$ (which
depends on $\overline{T}$) such that
$$
| \alpha^*(t) | \leq k' \|H\|_\infty\,, \quad t\in[0,T]\,.
$$ 
Using the fact that $\frac{\alpha^*}{\sin\alpha^*}\leq1 +
\frac{(\alpha^*)^2}{2}$ (since $\alpha^*<\frac{\pi}{2}$), we get
equations \eqref{eq:24}-\eqref{eq:25} from equations
\eqref{eq:control_turn_back} with $k$ larger than $1+
\frac{2\kappa}{M_2}k'$ and $1 +
\frac{M_1+M_2}{M_1M_2}k'^2\overline{H}$.
\end{proof}
\smallskip The proof of Lemma~\ref{lem:linearise} requires a more
technical lemma:
\begin{lemma}
  \label{lem:rang}
  Let Assumption \ref{ass} hold.  There exist arbitrarily small values
  of $\beta$ such that for a certain $\overline{\alpha}>0$,
  % For any $\beta>0$, there exists $\beta^*$, $0<\beta^*\leq\beta$,
  % and $\overline{\alpha}>0$, such that
  the distribution spanned by the vector fields $\vecX_3$, $\vecX_4$,
  $[\vecX_3,\vecX_4]$ and $\vecX_5^{\beta}$, with
  \begin{eqnarray}
    \vecX_5^{\beta}=-\beta \left(M_2+M_1\cos\alpha\right)
    [\vecX_3,[\vecX_3,\vecX_4]] %\quad\quad\quad\quad\nonumber\\
    +\left(\kappa\alpha-\beta M_2\right)
    [\vecX_4,[\vecX_3,\vecX_4]]\,,\label{eq:27}
  \end{eqnarray}
  has rank 4 at all points in
  \begin{equation}
    \label{eq:20}
    \{\vecz=(x,y,\theta,\alpha)\in\R^4\,,\ |\alpha|<\overline{\alpha},\,\alpha\neq0\}\,.
  \end{equation}
\end{lemma}
\begin{proof}
  The determinant of $\vecX_3$, $\vecX_4$, $[\vecX_3$, $\vecX_4]$,
  $\vecX_5^{\beta}$ depends only on $\alpha$ and $\beta$ (and the
  fixed parameters of the system), We computed it using a symbolic
  computation software (Maple). It is a polynomial with respect to
  $\cos\alpha$ of degree at most 12, whose coefficients are affine
  with respect to $\alpha$ and $\beta$. Its leading coefficient does
  not depend on $\alpha$ or $\beta$ and is zero only if 
$\eta_i=\xi_i$
  or $\eta_i=4\,\xi_i$, for $i=1$ or $i=2$. In all these particular
  cases, and using \eqref{eq:9}-\eqref{eq:8}, this determinant is a
  polynomial in $\cos\alpha $ with degree less than 12 but at most 1.
  Hence, for arbitrarily small values of $\beta$, the analytic
  function of one variable $\alpha \mapsto \det(\vecX_3,\vecX_4,$
  $[\vecX_3,\vecX_4],\vecX_5^{\beta})$ is non zero which ensures the
  existence of $\overline{\alpha}$.
\end{proof}
\begin{proof}[Proof of Lemma~\ref{lem:linearise}]
  Define the matrices
  $
  \vecB_j(t)=\left(\frac{\mathrm{d}}{\mathrm{d}t}-\vecA(t)\right)^j\vecB(t)\,,
  $
  with $\vecA(t)$ and $\vecB(t)$ given by \eqref{eq:linear_system}.
  According to \cite[Theorem 1.18]{Coron56}, the linear system
  \eqref{eq:linear_system} is controllable on $[0,\tau]$ if there is
  at least one $t$, $0<t<\tau$, such that
  \begin{equation}
    \label{eq:condition_controllability}
    \mathcal{S}p(t):=\mbox{Span}\left\{ \vecB_j(t)\vecv;\quad\vecv\in \mathbb{R}^2;\quad j\geq0 \right\} = \mathbb{R}^4\,.
  \end{equation}

  By a simple computation, it turns out that, with the constant
  controls $H^*_{\parallel}\!(t)=0$ and $H^*_{\perp}\!(t)=\beta$,
  % and in the same coordinates we used to express the matrices
  % $\vecA(t),\vecB(t)$,
  the $i$\textsuperscript{th} column of $\vecB_j(t)$ is the column of
  coordinates of the vector field $\vecCijbeta i j \beta$ at point
  $\vecz^*(t)$, with
  \begin{eqnarray}
    \vecCijbeta i 0 \beta=\vecF_i\,,
    \ 
    \vecCijbeta i 1 \beta=[\vecF_0+\beta\,\vecF_1,\vecCijbeta i 0 \beta]\,,
    % \nonumber\\
    \vecCijbeta i 2 \beta=[\vecF_0+\beta\,\vecF_1,\vecCijbeta i 1 \beta] %\quad\quad\quad\quad
    \label{eq:19}
  \end{eqnarray}
  and so on (we do not need $j>2$). We claim that
  \begin{equation}
    \label{eq:21}\left.\!\!\!
      \begin{array}{l}
        \text{for any $\beta>0$, there exists $\beta^*$,
        $0<\beta^*\leq\beta$, and $\overline{\alpha}>0$, such that} \\
        \text{the distribution spanned by }\vecCijbeta 10 \beta,\vecCijbeta 20 \beta,
        \vecCijbeta 11 \beta,\vecCijbeta 21 \beta,
        \vecCijbeta 12 \beta,\vecCijbeta 22 \beta\\
        \text{has rank 4 at all points of \eqref{eq:20}.}
      \end{array}\right\}
  \end{equation}
  % \begin{equation}
  %   \label{eq:21}\left.
  %   \begin{array}{l}
  %     \text{for any $\beta>0$, there exists $\beta^*$,
  %     $0<\beta^*\leq\beta$,} \\
  %      \text{and $\overline{\alpha}>0$ such that the distribution spanned}\\
  %      \text{by } \vecCijbeta 10 \beta,\vecCijbeta 20 \beta,
  %       \vecCijbeta 11 \beta,\vecCijbeta 21 \beta,
  %       \vecCijbeta 12 \beta,\vecCijbeta 22 \beta
  %      \text{ has }\\
  %      \text{rank 4 at all points of \eqref{eq:20}.}
  %   \end{array}\right\}
  % \end{equation}
  This claim implies Lemma~\ref{lem:linearise}. Indeed, take $\beta^*$
  in the lemma to be the one given by \eqref{eq:21}. Along the
  trajectory
  $t\mapsto\vecz^{\beta^*}(t)=(x^*(t),y^*(t),\theta^*(t),\alpha^*(t))$,
  one has $\alpha^*(0)=0$, and this implies $\dot\alpha^*(0)\neq 0$
  because \eqref{eq:8} is satisfied, $H_\perp$ is nonzero, and a
  straightforward computation from \eqref{eq:dynamics} shows that
  \begin{equation*}
    % \label{eq:48}
    % \alpha=0\ \Rightarrow\
    \dot\alpha=3
    \frac{\left(
        3+4\frac{\ell_2}{\ell_1}+\frac{\eta_2{\ell_2}^{2}}{\eta_1{\ell_1}^{2}} \right) 
      M_1
      -
      \left(
        3+4\frac{\ell_1}{\ell_2}
        +\frac {\eta_1{\ell_1}^{2}}{\eta_2{\ell_2}^{2}}  \right) 
      M_2
    }{
      \ell_1\,\ell_2\,(\eta_1\ell_1+\eta_2\ell_2)
    }\,H_\perp\,.
  \end{equation*}
  when $\alpha=0$. Hence there exists $\bar t>0$ (that depends on
  $\beta^*$) such that
  \begin{equation}
    \label{eq:alpha_nonnull}
    0<t<\bar t\ \ \Rightarrow\ \ \alpha(t)\neq0 \text{ and }
    |\alpha(t)|<\overline\alpha
    \,.
  \end{equation}
  According to the above remark and if the claim holds,
  \eqref{eq:condition_controllability} holds for all $t$,
  $0<t<\bar t$; linear controllability on $[0,\tau]$ in the lemma
  follows from \eqref{eq:condition_controllability} at some positive
  $t$ smaller than $\min\{\bar t,\tau\}$.

  Let us now prove \eqref{eq:21}. Recall that
  \begin{eqnarray}
    \vecF_0+\beta\,\vecF_1=-\beta\left(M_2+M_1\cos\alpha\right) \vecX_3
    % \nonumber \\
    % \quad\quad\quad\quad\quad\quad
    +\left(\kappa\alpha-\beta M_2\right) \vecX_4\,.  \label{eq:46}
  \end{eqnarray}  
  According to \eqref{eq:12} and \eqref{eq:19}, and since $M_1\neq0$
  and $M_2\neq0$, $\vecCijbeta 1 0 \beta$ and $\vecCijbeta 2 0 \beta$
  span the same distribution as $\vecX_3$ and $\vecX_4$ at points
  where $\alpha\neq0$.  Hence
  $\vecCijbeta 1 0 \beta,\vecCijbeta 2 0 \beta,\vecCijbeta 1 1
  \beta,\vecCijbeta 2 1 \beta$
  span the same distribution as $\vecX_3$, $\vecX_4$,
  $[\vecF_0+\beta\,\vecF_1,\vecX_3]$ and
  $[\vecF_0+\beta\,\vecF_1,\vecX_4]$, that is, according to
  \eqref{eq:46}, $\vecX_3$, $\vecX_4$,
  $\left(\beta M_2-\kappa\alpha\right) [\vecX_3,\vecX_4]$, and
  $-\beta\left(M_2+M_1\cos\alpha\right) [\vecX_3,\vecX_4]$ i.e., if,
  in addition, $(M_2+M_1\cos\alpha,$
  $\kappa\alpha-\beta M_2)\neq(0,0)$, the same distribution as
  $\{\vecX_3,\,\vecX_4,\,[\vecX_3,\vecX_4]\}$, and finally
  $\vecCijbeta 1 0 \beta$, $\vecCijbeta 2 0 \beta$,
  $\vecCijbeta 1 1 \beta$, $\vecCijbeta 2 1 \beta$,
  $ \vecCijbeta 1 2 \beta$, $\vecCijbeta 2 2 \beta$ span the same
  distribution as $\vecX_3$, $\vecX_4$, $[\vecX_3,\vecX_4]$ and
  $\vecX_5^\beta$ given by \eqref{eq:27} at points where $\alpha\neq0$
  and $(M_2+M_1\cos\alpha,\kappa\alpha-\beta M_2)\neq(0,0)$.  This
  property, together with Lemma~\ref{lem:rang}, proves the claim
  \eqref{eq:21}, hence Lemma~\ref{lem:linearise}.
\end{proof}

\section{Perspectives}
\label{sec:Perspectives}

For a micro-swimmer
model made by two magnetized segments connected by an elastic joint,
controlled by an external magnetic field,
this note establishes local controllability \emph{around the straight
  position} by controls that cannot be made
arbitrarily small, but an explicit bound on the controls is given.
This raises two natural questions.

On the one hand one would like to decide whether
 the controls can be taken arbitrarily small, thus proving STLC in the
  sense of \cite{Coron56}, or the (nonzero)
bound is sharp, and hence the system is \emph{not} STLC.
On
the other hand, since controllability around non-straight positions is
easier, it is natural to address the question of global
controllability as an extension of this note.
  
Further perspectives, that are currently under our investigation, are
to extend our result to more realistic swimmers, for instance by
considering additional segments. An other important issue is to study
the optimal control problem i.e., finding the magnetic fields in such
a way that the swimmer reaches a desired configuration as quick as
possible; this has already been done for swimmers that are controlled
by the velocity at each of their joints \cite{cdc2013} but the study
for the present system where theses velocities are indirectly
controlled by a magnetic field is quite different.

\section*{Aknowlegement}
% Work supported in part by the French Space Agency CNES, R\&T action
% R-S13/BS-005-012 and by the region Provence-Alpes-C\^ote d'Azur.
% Laetitia Giraldi was funded by the labex LMH through the grant
% ANR-11-LABX-0056-LMH in the ``Programme des Investissements
% d'Avenir''.
The authors thank Cl{\'e}ment Moreau, student at ENS Cachan, France,
for carefully reading the manuscript.

% \bibliographystyle{plain} \bibliography{2_Link_Magnetic}

\begin{thebibliography}{10}

\bibitem{AlougesDeSimone13} F.~Alouges, A.~DeSimone, L.~Giraldi, and
  M.~Zoppello, ``Self-propulsion of slender micro-swimmers by
  curvature control: {N}-link swimmers,'' \emph{Journal of Non-Linear
    Mechanics}, 2013.

\bibitem{AlougesDeSimone14} ------, ``Can magnetic multilayers propel
  micro-swimmers mimicking sperm cells?'' \emph{Soft Robotics},
  vol.~2, no.~3, pp. 117--128, 2015.

\bibitem{AlougesGiraldi12} F.~Alouges and L.~Giraldi, ``Enhanced
  controllability of low {R}eynolds number swimmers in the presence of
  a wall,'' \emph{Acta Appl. Math.}, vol. 128, no.~1, pp. 153--179,
  2013.

\bibitem{BettiolBonnard15} P.~Bettiol, B.~Bonnard, L.~Giraldi,
  P.~Martinon, and J.~Rouot, ``\href{http://hal.inria.fr/hal-01143763}{The purcell three-link swimmer: some
  geometric and numerical aspects related to periodic optimal
  controls},'' Rad. Ser. Comp. App. \textbf{18}, 2015.

\bibitem{BrennerHappel65} H.~Brenner and J.~Happel, \emph{Low
    {R}eynolds number hydrodynamics: with special applications to
    particulate media.}\hskip 1em plus 0.5em minus 0.4em\relax
  Springer, 1965.

\bibitem{ChambrionGiraldi14} T.~Chambrion, L.~Giraldi, and A.~Munnier,
  ``Optimal strokes for driftless swimmers: A general geometric
  approach,'' \emph{hal-00969259}, 2014.

\bibitem{Coron92} J.-M. Coron, ``Global asymptotic stabilization of
  controllable systems without drift,'' \emph{Math. Control Signals
    Syst.}, vol.~5, no.~3, pp.  295--312, 1992.

\bibitem{Coron56} ------, \emph{Control and nonlinearity},
  ser. Mathematical Surveys and Monographs.\hskip 1em plus 0.5em minus
  0.4em\relax American Math.  Society, Providence, RI, 2007, vol. 136.

\bibitem{Gerard-VaretGiraldi13} D.~G{\'e}rard-Varet and L.~Giraldi,
  ``Rough wall effect on microswimmer,'' \emph{ESAIM: COCV}, vol.~21,
  no.~3, 2015.

\bibitem{cdc2013} L.~Giraldi, P.~Martinon, and M.~Zoppello,
  ``Controllability and optimal strokes for {N}-link microswimmer,''
  in \emph{52nd IEEE Conf. on Decision and Control}, Dec. 2013,
  pp. 3870--3875.

\bibitem{GrayHancock55} J.~Gray and J.~Hancock, ``The propulsion of
  sea-urchin spermatozoa,'' \emph{J.  Exp. Biol.}, vol.~32, no.~4,
  pp. 802--814, 1955.

\bibitem{Or14} E.~Gutman and Y.~Or, ``Simple model of a planar
  undulating magnetic microswimmer,'' \emph{Phys. Rev. E}, vol.~90,
  no. 013012, 2014.

\bibitem{PowersLauga09} E.~Lauga and T.~Powers, ``The hydrodynamics of
  swimming microorganisms,'' \emph{Rep. Prog. Phys.}, vol.~72,
  no. 09660, 2009.

\bibitem{LoheacScheid11} J.~Loh{\'e}ac, J.~F. Scheid, and M.~Tucsnak,
  ``Controllability and time optimal control for low {R}eynolds
  numbers swimmers,'' \emph{Acta Appl. Math.}, vol.  123, no.~1,
  pp. 175--200, 2013.

% \bibitem{LoheacMunnier12} J.~Loh{\'e}ac and A.~Munnier,
%   ``Controllability of 3{D} low {R}eynolds number swimmers,''
%   \emph{ESAIM Control Optim. Calc. Var.}, vol.~20, no.~1, pp.
%   236--268, 2014.

\bibitem{Montgomery02} R.~Montgomery, \emph{A tour of subriemannian
    geometries, theirs geodesics and applications}.\hskip 1em plus
  0.5em minus 0.4em\relax American Mathematical Society, 2002.

\bibitem{NelsonKaliakatsos10} B.~J. Nelson, I.~K. Kaliakatsos, and
  J.~J. Abbott, ``Microrobots for minimally invasive medicine,''
  \emph{Annu. Rev. Biomed. Eng.}, vol.~12, pp. 55--85, 2010.

\bibitem{Purcell77} E.~M. Purcell, ``Life at low {R}eynolds number,''
  \emph{Am. J. Phys.}, vol.~45, pp. 3--11, 1977.

\bibitem{Shapere89} A.~Shapere and F.~Wilczek, ``Efficiencies of
  self-propulsion at low {R}eynolds number,'' \emph{J. Fluid Mech.},
  1989.

\bibitem{Suss87siam} H.~J. Sussmann, ``A general theorem on local
  controllability,'' \emph{SIAM J.  Control Optim.}, vol.~25, no.~1,
  pp. 158--194, 1987.
  
\end{thebibliography}

% \appendix
% \section{Rank condition}
% \label{appendix}
\end{document}